\documentclass[pdflatex,sn-mathphys-num]{sn-jnl}

\usepackage{amsmath,amssymb,amsthm,mathtools,mathrsfs}
\usepackage{graphicx}%
\usepackage{multirow}%
\usepackage{amsmath,amssymb,amsfonts}%
\usepackage{amsthm}%
\usepackage{mathrsfs}%
\usepackage[title]{appendix}%
\usepackage{xcolor}%
\usepackage{textcomp}%
\usepackage{manyfoot}%
\usepackage{booktabs}%
\usepackage{algorithm}%
\usepackage{algorithmicx}%
\usepackage{algpseudocode}%
\usepackage{listings}%
\usepackage{xcolor}
\usepackage{graphicx}
\usepackage{booktabs}
\usepackage{hyperref}
\hypersetup{
  colorlinks=true,
  linkcolor=blue!55!black,
  citecolor=blue!55!black,
  urlcolor=blue!55!black
}

\theoremstyle{plain}
\newtheorem{theorem}{Theorem}[section]
\newtheorem{proposition}[theorem]{Proposition}
\newtheorem{lemma}[theorem]{Lemma}
\newtheorem{corollary}[theorem]{Corollary}
\theoremstyle{definition}
\newtheorem{definition}[theorem]{Definition}
\newtheorem{assumption}[theorem]{Assumption}
\newtheorem{remark}[theorem]{Remark}

\newcommand{\R}{\mathbb{R}}
\newcommand{\E}{\mathbb{E}}
\newcommand{\PP}{\mathbb{P}}
\newcommand{\F}{\mathbb{F}}
\newcommand{\Rp}{\R_+^2}
\newcommand{\Rpp}{\R_{++}^2}
\newcommand{\Lc}{\mathcal{L}}
\newcommand{\Cc}{\mathcal{C}}
\newcommand{\Dc}{\mathcal{D}}
\newcommand{\Tc}{\mathcal{T}}
\newcommand{\Fc}{\mathcal{F}}
\newcommand{\sgD}{\sigma_\Delta}
\newcommand{\Gtot}{\Gamma^{\mathrm{tot}}}
\newcommand{\Gac}{\Gamma^{\mathrm{ac}}}
\newcommand{\GD}{\Gamma^{\Delta}}
\newcommand{\Gbd}{\Gamma^{\partial}}
\newcommand{\Lac}{\Lc^{\mathrm{ac}}}
\newcommand{\bb}{\mathfrak{b}}
\newcommand{\Tr}{\operatorname{Tr}}
\newcommand{\tr}{\operatorname{tr}}
\newcommand{\bone}{\mathbf 1}

\begin{document}
\title{A Measure-Valued Obstacle Problem for an Obliquely Reflected Diffusion with a Max-Type Payoff}

\author*[1]{\fnm{Louis Shuo} \sur{Wang}}
\email{wang.s41@northeastern.edu}
\equalcont{These authors contributed equally to this work.}

\author[2]{\fnm{Jiguang} \sur{Yu}}
\email{jyu678@bu.edu}
\equalcont{These authors contributed equally to this work.}

\author*[3]{\fnm{Ye} \sur{Liang}}\email{ye-liang@uiowa.edu}
\equalcont{These authors contributed equally to this work.}

\affil[1]{\orgdiv{Department of Mathematics},
\orgname{Northeastern University},
\orgaddress{
\city{Boston},
\state{MA},
\postcode{02115},
\country{USA}
}}

\affil[2]{\orgdiv{College of Engineering},
\orgname{Boston University},
\orgaddress{
\city{Boston},
\state{MA},
\postcode{02215},
\country{USA}
}}

\affil[3]{\orgdiv{College of Engineering},
\orgname{The University of Iowa},
\orgaddress{
\city{Iowa City}, 
\postcode{52242}, 
\state{IA},
\country{USA}}}
\abstract{
We study an obliquely reflected optimal stopping problem in the nonnegative
quadrant with nonsmooth max-type payoff \(G(x)=x_1\vee\alpha x_2\), and we
develop a measure-valued potential-theoretic formulation of the associated
obstacle problem.  The kink of \(G\) on the diagonal \(x_1=\alpha x_2\)
produces a singular surface measure in the distributional generator, while the
oblique reflection directions generate boundary local-time contributions on
the coordinate faces.  Together with the absolutely continuous stopping gain,
these terms define a total signed stopping measure \(\Gtot\).  We derive the
corresponding reflected It\^{o}--Tanaka identity, prove a killed-resolvent
representation of the value function in the continuation region, and show that
the unrestricted reflected resolvent is generally incorrect because the
process is not absorbed on the stopping set.  The free boundary is formulated
through a continuation-side trace condition for the killed potential.  Under a
vertical monotonicity hypothesis on \(V-G\), the stopping set is shown to have
an epigraph form.  We finally prove a verification theorem: any admissible
epigraph candidate satisfying contact, strict continuation, reflected Neumann
compatibility, growth, the trace condition, and measure-superharmonicity
coincides with the value function, and its first entry time is optimal.
}



\maketitle

\section{Introduction}
\label{sec:introduction}

Let $X=(X^1,X^2)$ be an obliquely reflected diffusion in the nonnegative quadrant
$\Rp=[0,\infty)^2$,
\begin{equation}
\label{eq:intro-sde}
        dX_t=\beta(X_t)\,dt+\sigma(X_t)\,dW_t
        +\rho_1\,dL_t^1+\rho_2\,dL_t^2,
        \qquad X_0=x\in\Rp .
\end{equation}
where $L^1$, $L^2$ are the boundary reflection processes, the reflection on the
face $F_i:=\{x_i=0\}$ acting along a fixed inward vector $\rho_i$, and
$a:=\sigma\sigma^\top$. For $r>0$, a running cost $c\ge0$, and the max-type
reward
\[
        G(x)=x_1\vee\alpha x_2,\qquad \alpha>0,
\]
we study
\begin{equation}
\label{eq:intro-value}
        V(x)=\sup_{\tau}\E_x\!\left[
        e^{-r\tau}G(X_\tau)-\int_0^\tau e^{-rs}c(X_s)\,ds\right].
\end{equation}
Formally, $V$ solves the reflected obstacle problem
\begin{equation}
\label{eq:intro-obstacle}
        \max\{\Lc V-rV-c,\;G-V\}=0\quad\text{in }\Rp,
        \qquad \rho_i\!\cdot\!\nabla V=0\quad\text{on }F_i.
\end{equation}
The purpose of this paper is to give a measure-valued potential-theoretic
formulation of \eqref{eq:intro-obstacle} and a verification theorem for
candidate free boundaries. We deliberately do not claim to solve the free
boundary; we prove a rigorous verification framework.

This work belongs to the intersection of optimal stopping, reflected
diffusions, and measure-valued potential theory.  The connection between
optimal stopping and obstacle-type variational inequalities is classical
\cite{el2006aspects,shiryaev2008optimal,peskir2006optimal,wang2025analysis,bensoussan2011applications,friedman1982variational,wang2025analysis1,kinderlehrer2000introduction,talbi2023dynamic,liu2025bidirectional,liang2025global,talbi2023viscosity,colaneri2022class,wang2026algebraic,ekstrom2026dynkin}.  Reflected diffusions
and the Skorokhod problem, including oblique reflection in orthants, are treated
in 
\cite{skorokhod1961stochastic,tanaka1979stochastic,wang2026damage,lions1984stochastic,stroock1971diffusion,dupuis1993sdes,yu2026pattern,kang2017submartingale,lipshutz2018directional,pilipenko2024boundary,wang2026breakdown,dianetti2023multidimensional,harrison1981reflected,yu2026rigorous,varadhan1985brownian,reiman1988boundary,taylor1993existence,williams1995semimartingale,yu2026beyond}.
Reflected obstacle problems are naturally
formulated through constrained viscosity solutions with Neumann or oblique
boundary conditions \cite{crandall1992user,lions1985neumann,barles1994solutions,
ishii1991fully,soner1986optimal}.  Our analysis uses the It\^{o}--Tanaka
formula and the potential theory of additive functionals and Revuz measures
\cite{revuz2013continuous,karatzas2014brownian,protter2012stochastic,cai2026optimal,peskir2005change,revuz1970mesures,blumenthal2007markov,dynkin1965markov,wang2025multi,fukushima2011dirichlet,marcus2006markov}.
Multidimensional stopping and the geometry of the
resulting exercise regions has been studied chiefly in the option-pricing
literature \cite{broadie1997valuation,villeneuve1999exercise}; questions of free-boundary
regularity in such problems remain delicate
\cite{dayanik2003optimal,gao2022rolling,de2020global,laurence2009regularity,yu2026from,ke2022parallel,soner2025stopping}. 
The novelty is to combine the diagonal kink measure
generated by \(G=x_1\vee\alpha x_2\) with the boundary measure generated by
oblique reflection, and to formulate the free-boundary condition as a
continuation-side trace condition for the corresponding killed potential.

The analytical difficulty is the simultaneous presence of reflection, killing
at the stopping set, and the nonsmooth obstacle $G=x_1\vee\alpha x_2$. Let
\[
        \Delta:=\{x\in\Rp:x_1=\alpha x_2\},\qquad Y(x):=x_1-\alpha x_2.
\]
Then $G(x)=\alpha x_2+Y(x)^+$, so $G$ is affine on the two components of
$\Rp\setminus\Delta$ but its gradient jumps across $\Delta$. In distributions,
\begin{equation}
\label{eq:intro-D2G}
        D^2G=\frac{(1,-\alpha)^\top(1,-\alpha)}{\sqrt{1+\alpha^2}}\,\sgD,
\end{equation}
where $\sgD$ is one-dimensional surface measure on $\Delta$. Consequently
$\Lc G$, and hence the stopping gain associated with the obstacle, is not a
function but a signed measure carrying a singular component on $\Delta$.

Reflection adds a second mechanism. Because $G$ need not satisfy the reflected
Neumann condition $\rho_i\!\cdot\!\nabla G=0$ on the coordinate faces, the
reflected It\^{o}--Tanaka formula for $G(X)$ contains boundary local-time
terms. The correct object is therefore the total signed stopping
measure
\begin{equation}
\label{eq:intro-total}
        \Gtot=\Gac\,dx+\GD+\Gbd,
\end{equation}
whose three parts are the absolutely continuous interior gain, the diagonal
singular measure, and the signed boundary Revuz measure generated by the
reflection contribution; the precise definitions are given in
Section~\ref{sec:measure-gain}. The associated signed additive functional
$A^{\Gtot}$ produces the bridge identity
\begin{equation}
\label{eq:intro-bridge}
        e^{-rt}G(X_t)-\int_0^t e^{-rs}c(X_s)\,ds
        =G(x)-\int_0^t e^{-rs}\,dA_s^{\Gtot}+M_t^G,
\end{equation}
up to localization, with $M^G$ a local martingale. This identity is the
analytic starting point of the paper.

Let $H:=V-G$, $\Cc:=\{H>0\}$, $\Dc:=\{H=0\}$, and
$\tau_{\Dc}:=\inf\{t\ge0:X_t\in\Dc\}$. The potential representation must be
killed at $\tau_{\Dc}$. For a signed smooth measure $\mu$ define the killed
$r$-resolvent on $\Cc$,
\begin{equation}
\label{eq:intro-killed}
        R_r^{\Cc}\mu(x):=\E_x\!\left[\int_0^{\tau_{\Dc}}e^{-rs}\,dA_s^\mu\right].
\end{equation}
If $\tau_{\Dc}$ is optimal and the additive functionals are integrable, then
\begin{equation}
\label{eq:intro-main}
        V(x)=G(x)-R_r^{\Cc}\Gtot(x),\qquad x\in\Cc.
\end{equation}
The unrestricted reflected resolvent $R_r^{\mathrm R}(\Gtot\bone_{\Cc})$ is
generally wrong: a non-absorbed process re-enters $\Cc$ after $\tau_{\Dc}$, so
the unrestricted potential counts occupation that has no variational meaning.
The free boundary is then characterized by the continuation-side trace
condition
\begin{equation}
\label{eq:intro-trace}
        \Tr_{\partial\Cc}R_r^{\Cc}\Gtot=0,
\end{equation}
which, when $\Dc$ is an epigraph $\{x_2\ge\bb(x_1)\}$, reads
$\lim_{\varepsilon\downarrow0}R_r^{\Cc}\Gtot(x_1,\bb(x_1)-\varepsilon)=0$ at
boundary points where the one-sided trace exists.

\paragraph{Contributions} The main contribution of this paper is a measure-valued killed-resolvent
formulation for an obliquely reflected obstacle problem with the nonsmooth
payoff \(G(x)=x_1\vee\alpha x_2\).  The kink of \(G\) produces a diagonal
surface measure, while the oblique reflection directions produce boundary
Revuz measures.  We combine these terms into a total signed stopping measure
\(\Gamma^{\rm tot}\), prove the associated reflected It\^{o}--Tanaka bridge
identity, and show that the value admits a killed-resolvent representation in
the continuation region.  We further show that the unrestricted reflected
resolvent is generally incorrect, formulate the free boundary through a
continuation-side trace condition, and prove a verification theorem for
epigraph candidates satisfying contact, strict continuation, reflected Neumann
compatibility, growth, trace, and measure-superharmonicity.

The paper is organized as follows.
Section~\ref{sec:problem} formulates the reflected stopping problem and its
obstacle inequality. Section~\ref{sec:measure-gain} derives the distributional
generator of the max payoff and the reflected It\^{o}--Tanaka formula.
Section~\ref{sec:killed-resolvent} introduces the total stopping measure, proves
the killed-resolvent representation, and explains the failure of the
unrestricted resolvent. Section~\ref{sec:free-boundary-trace} gives the
free-boundary trace formulation. Section~\ref{sec:verification} states the
admissibility conditions and proves the verification theorem.
Section~\ref{sec:conclusion} concludes.

\section{Problem formulation and reflected variational inequality}
\label{sec:problem}

\subsection{Reflected diffusion in the quadrant}
\label{subsec:reflected-diffusion}

Write $\Rp=[0,\infty)^2$ and $\Rpp=(0,\infty)^2$. Let $W=(W^1,W^2)$ be a
two-dimensional Brownian motion on a filtered probability space
$(\Omega,\Fc,\F,\PP)$, $\F=(\Fc_t)_{t\ge0}$, satisfying the usual conditions.
We consider the reflected diffusion \eqref{eq:intro-sde} in $\Rp$, where
\[
        \beta:\Rp\to\R^2,\qquad\sigma:\Rp\to\R^{2\times2},\qquad a:=\sigma\sigma^\top.
\]
The cumulative reflection term is
\[
        K_t:=\rho_1L_t^1+\rho_2L_t^2,
\]
so that the reflected SDE may equivalently be written as
\[
        dX_t=\beta(X_t)\,dt+\sigma(X_t)\,dW_t+dK_t .
\]
Each $L^i$ is continuous, adapted, nondecreasing, satisfies $L_0^i=0$, and
obeys the support condition
\begin{equation}
\label{eq:reflection-support}
        \int_0^\infty\bone_{\{X_s^i>0\}}\,dL_s^i=0,\qquad i=1,2,
\end{equation}
so that $L^i$ increases only on the face $F_i:=\{x\in\Rp:x_i=0\}$. We fix the
constant inward reflection vectors
\begin{equation}
\label{eq:reflection-vectors}
        \rho_1=(1,\theta_1),\qquad\rho_2=(\theta_2,1),
        \qquad \theta_1,\theta_2\in\R,
\end{equation}
each normalized so that its $e_i$-component equals $1$; the inward condition
$\rho_i\cdot e_i>0$ holds automatically. The case
$\theta_1=\theta_2=0$, i.e. $\rho_i=e_i$, is normal reflection; the
general case is oblique reflection. For the present payoff and the
reflection vectors $\rho_1=(1,\theta_1)$, $\rho_2=(\theta_2,1)$, the boundary
contribution to the stopping measure vanishes exactly when $\theta_1=\theta_2=0$.

For $f\in C^2(\Rp)$ the interior generator is
\begin{equation}
\label{eq:generator}
        \Lc f(x)=\beta(x)\cdot\nabla f(x)+\tfrac12\tr\!\big(a(x)D^2f(x)\big),
        \qquad x\in\Rpp,
\end{equation}
and the reflected It\^{o} formula reads
\begin{equation}
\label{eq:reflected-ito}
        f(X_t)=f(x)+\int_0^t\Lc f(X_s)\,ds
        +\sum_{i=1}^2\int_0^t(\rho_i\!\cdot\!\nabla f)(X_s)\,dL_s^i
        +\int_0^t\nabla f(X_s)\sigma(X_s)\,dW_s .
\end{equation}
The finite-variation boundary term vanishes for every $t$ if and only if $f$
satisfies the reflected Neumann condition
\begin{equation}
\label{eq:reflected-neumann}
        \rho_i\!\cdot\!\nabla f=0\qquad\text{on }F_i,\quad i=1,2,
\end{equation}
which for normal reflection is the classical condition $\partial_i f=0$ on
$F_i$. This point is essential below, since the payoff $G=x_1\vee\alpha x_2$
does not satisfy \eqref{eq:reflected-neumann} when the reflection is oblique.

\begin{assumption}[Coefficients and well-posedness]
\label{ass:diffusion}
The coefficients $\beta,\sigma$ are continuous and locally Lipschitz on $\Rp$
with linear growth $|\beta(x)|+\|\sigma(x)\|\le K(1+|x|)$, and
$a=\sigma\sigma^\top$ is locally uniformly elliptic in $\Rpp$. We assume
\[
        1-\theta_1\theta_2>0,
\]
a standard sufficient completely-$\mathcal S$ condition
\cite{harrison1981reflected,taylor1993existence} for the orthant
reflection matrix
\[
        R=\begin{pmatrix}1&\theta_2\\\theta_1&1\end{pmatrix},
\]
so that the associated orthant Skorokhod problem is well posed; for every
$x\in\Rp$ the reflected SDE \eqref{eq:intro-sde} admits a unique strong
solution in $\Rp$, and the associated transition semigroup is Feller.
Moreover, the boundary local times $L^1,L^2$ do not charge the corner $\{0\}$.
\end{assumption}

\subsection{The stopping problem}
\label{subsec:stopping-problem}

Fix $r>0$, $\alpha>0$, and a continuous running cost $c:\Rp\to[0,\infty)$. The
obstacle is $G(x)=x_1\vee\alpha x_2$, and the value function is
\eqref{eq:intro-value}, the supremum being taken over the set $\Tc$ of all
$\F$-stopping times. Since $\tau=0$ is admissible,
\begin{equation}
\label{eq:V-dominates-G}
        V\ge G\qquad\text{on }\Rp.
\end{equation}
We define the stopping advantage, continuation set, and stopping set by
\begin{equation}
\label{eq:HCD}
        H:=V-G,\qquad \Cc:=\{H>0\},\qquad \Dc:=\{H=0\},
\end{equation}
and the first entry time $\tau_{\Dc}:=\inf\{t\ge0:X_t\in\Dc\}$.

\begin{assumption}[Growth and integrability]
\label{ass:growth}
There exist $\Psi\in C^2(\Rp)$, constants $K_\Psi,\lambda_\Psi>0$, and $m\ge1$
such that
\[
        1+|x|^m\le K_\Psi\Psi(x),\qquad G(x)+c(x)\le K_\Psi\Psi(x),
\]
\begin{equation}
\label{eq:lyapunov}
        \Lc\Psi(x)\le\lambda_\Psi\Psi(x)+K_\Psi,\qquad x\in\Rpp,
\end{equation}
together with the boundary compatibility $\rho_i\!\cdot\!\nabla\Psi\le0$ on
$F_i$, $i=1,2$, the discount condition $r>\lambda_\Psi$, and the integrability
\[
        \E_x\Big[\sup_{t\ge0}e^{-rt}\Psi(X_t)\Big]<\infty,\qquad x\in\Rp .
\]
\end{assumption}

Applying \eqref{eq:reflected-ito} to $\Psi$ and using
$\rho_i\!\cdot\!\nabla\Psi\le0$ to discard the boundary term, followed by
Gronwall's inequality, gives
\begin{equation}
\label{eq:Psi-bound}
        \E_x[\Psi(X_t)]\le C(1+\Psi(x))e^{\lambda_\Psi t},
\end{equation}
hence
$\int_0^\infty e^{-rt}\E_x[\Psi(X_t)]\,dt\le C(1+\Psi(x))$. In particular the
cost integral is finite and $V$ has at most $\Psi$-growth.

\begin{proposition}[Finiteness and regularity]
\label{prop:finiteness}
Under Assumptions~\ref{ass:diffusion} and \ref{ass:growth}, $V$ is finite and
$G(x)\le V(x)\le C(1+\Psi(x))$. Moreover $V$ is lower semicontinuous; if the
reflected semigroup is locally strong Feller, then $V$ is continuous, and
$\Dc=\{V=G\}$ is closed.
\end{proposition}

\begin{proof}
The lower bound is \eqref{eq:V-dominates-G}. For the upper bound, by the
reflected It\^{o} formula and the boundary compatibility
$\rho_i\!\cdot\!\nabla\Psi\le0$,
\[
        e^{-rt}\Psi(X_t)\le\Psi(x)
        +\int_0^t e^{-rs}\big[-(r-\lambda_\Psi)\Psi(X_s)+K_\Psi\big]\,ds+M_t.
\]
After localization, the process
\[
        e^{-rt}\Psi(X_t)-\int_0^t e^{-rs}K_\Psi\,ds
\]
is bounded above by a local martingale plus the initial value \(\Psi(x)\).
Taking expectations and then letting the localization parameter tend to
infinity gives
\[
        \E_x[e^{-r\tau}\Psi(X_\tau)]
        \le
        \Psi(x)+\frac{K_\Psi}{r}
\]
for every bounded stopping time \(\tau\), and then for all stopping times by
monotone convergence.
Since $G+c\le K_\Psi\Psi$ and $c\ge0$,
\[
        V(x)\le\sup_{\tau}\E_x\big[e^{-r\tau}G(X_\tau)\big]
        \le K_\Psi\Big(\Psi(x)+\frac{K_\Psi}{r}\Big)\le C(1+\Psi(x))<\infty .
\]
The integrability hypothesis of Assumption~\ref{ass:growth} ensures, in
addition, the uniform integrability used in the limiting arguments below. Lower
semicontinuity follows by fixing
a stopping time, using the Feller property and continuity of $G,c$, and taking
the supremum. Under the local strong Feller hypothesis, stability of
finite-horizon approximations and monotone convergence give continuity, whence
$\Dc=\{V-G=0\}$ is closed.
\end{proof}

\subsection{Dynamic programming and the reflected variational inequality}
\label{subsec:dpp}

For $T>0$ let $V_T(x)=\sup_{\tau\le T}\E_x[e^{-r\tau}G(X_\tau)
-\int_0^\tau e^{-rs}c(X_s)\,ds]$; then $V_T\uparrow V$. The strong Markov
property gives, for every stopping time $\rho$,
\begin{equation}
\label{eq:dpp}
        V(x)=\sup_{\tau\in\Tc}\E_x\!\left[
        e^{-r(\tau\wedge\rho)}V(X_{\tau\wedge\rho})
        -\int_0^{\tau\wedge\rho}e^{-rs}c(X_s)\,ds\right].
\end{equation}

\begin{proposition}[Dynamic programming]
\label{prop:dpp}
Under Assumptions~\ref{ass:diffusion}--\ref{ass:growth}, $V$ satisfies
\eqref{eq:dpp}; consequently $V\ge G$ and $\Lc V-rV-c\le0$ in the viscosity
sense in $\Rpp$, with $\Lc V-rV-c=0$ in $\Cc$.
\end{proposition}

\begin{proof}
The finite-horizon values satisfy the dynamic programming principle by the
strong Markov property, and the estimates following
Assumption~\ref{ass:growth} justify passing to the infinite-horizon limit.
Testing \eqref{eq:dpp} against small exit times from balls in $\Rpp$ gives the
two viscosity inequalities; on the open set $\Cc=\{V>G\}$ the obstacle is
inactive and they combine to the continuation equation.
\end{proof}

Thus $V$ is a constrained viscosity solution of
\begin{equation}
\label{eq:reflected-obstacle}
        \max\{\Lc V-rV-c,\;G-V\}=0\quad\text{in }\Rpp,
\end{equation}
with the reflection condition imposed on the coordinate faces. We record the
boundary condition only in the following minimal form, which is all that the
present paper requires.

\begin{remark}[Boundary condition; no comparison principle used]
\label{rem:boundary-condition}
On the faces $F_i$ the reflection condition is understood in the constrained
viscosity sense associated with reflection along $\rho_i$, in the relaxed
(disjunctive) Barles--Lions form. Whenever $V$ is differentiable at a boundary
point, this condition reduces to the reflected Neumann condition
$\rho_i\!\cdot\!\nabla V=0$, $i=1,2$. We shall not invoke a comparison
principle for the reflected viscosity problem: the identification of $V$ is
obtained entirely from the reflected It\^{o}--Tanaka formula and the
verification theorem of Section~\ref{sec:verification}.
\end{remark}

\begin{theorem}[Reflected obstacle characterization]
\label{thm:reflected-obstacle}
Under Assumptions~\ref{ass:diffusion}--\ref{ass:growth}, $V$ is a constrained
viscosity solution of \eqref{eq:reflected-obstacle} with reflection condition
$\rho_i\!\cdot\!\nabla V=0$ on $F_i$ in the sense of
Remark~\ref{rem:boundary-condition}, and
\[
        V=G\ \text{on }\Dc,\qquad \Lc V-rV-c=0\ \text{in }\Cc
\]
in the viscosity sense.
\end{theorem}

\begin{proof}
The interior statement is Proposition~\ref{prop:dpp}. At boundary points the
reflected It\^{o} formula \eqref{eq:reflected-ito} applied to a test function
contains the term $\sum_i\int(\rho_i\!\cdot\!\nabla\varphi)\,dL^i$, supported on
$F_i$; the dynamic programming argument over small stopped intervals yields the
relaxed reflection alternatives, which at differentiability points give
$\rho_i\!\cdot\!\nabla V=0$. The contact identity is the definition of $\Dc$,
and the continuation equation holds because the obstacle is inactive on $\Cc$.
\end{proof}

\section{Distributional generator and reflected It\^{o}--Tanaka formula}
\label{sec:measure-gain}

\subsection{Distributional generator of the max payoff}
\label{subsec:distributional-generator}

Set $\Delta=\{x_1=\alpha x_2\}$, $E_1=\{x_1>\alpha x_2\}$,
$E_2=\{x_1<\alpha x_2\}$. On $E_1$ and $E_2$ respectively $G=x_1$ and
$G=\alpha x_2$, so $\nabla G=e_1$ on $E_1$ and $\nabla G=\alpha e_2$ on $E_2$.
Writing $Y(x)=x_1-\alpha x_2$ we have $G(x)=\alpha x_2+Y(x)^+$, and since
$d^2(y^+)/dy^2=\delta_0$,
\begin{equation}
\label{eq:D2G-measure}
        D^2G=\frac{(1,-\alpha)^\top(1,-\alpha)}{\sqrt{1+\alpha^2}}\,\sgD,
\end{equation}
where we used
$\delta_{\{x_1-\alpha x_2=0\}}=|\nabla(x_1-\alpha x_2)|^{-1}\sgD
=(1+\alpha^2)^{-1/2}\sgD$. Therefore, in the distributional sense on $\Rpp$,
\begin{equation}
\label{eq:LG-measure}
        \Lc G(dx)=\Lac G(x)\,dx+\frac{q(x)}{2\sqrt{1+\alpha^2}}\,\sgD(dx),
        \qquad \Lac G:=\beta_1\bone_{E_1}+\alpha\beta_2\bone_{E_2},
\end{equation}
with
\begin{equation}
\label{eq:q-def}
        q(x):=(1,-\alpha)\,a(x)\,(1,-\alpha)^\top>0
        \quad\text{on compacts of }\Rpp.
\end{equation}

Define the preliminary interior stopping gain $\Gamma:=c+rG-\Lc G$. Then
$\Gamma=\Gac\,dx+\GD$, where
\begin{equation}
\label{eq:Gamma-ac}
        \Gac(x)=c(x)+rG(x)-\beta_1(x)\bone_{E_1}(x)-\alpha\beta_2(x)\bone_{E_2}(x),
\end{equation}
\begin{equation}
\label{eq:Gamma-delta}
        \GD(dx)=-\frac{q(x)}{2\sqrt{1+\alpha^2}}\,\sgD(dx).
\end{equation}
The sign is negative because the singular part of $\Lc G$ is positive and
$\Gamma=c+rG-\Lc G$.

\begin{proposition}[Distributional stopping gain]
\label{prop:distributional-gain}
Under Assumption~\ref{ass:diffusion}, $c+rG-\Lc G$ is a signed Radon measure on
compact subsets of $\Rpp$, with absolutely continuous and diagonal singular
parts given by \eqref{eq:Gamma-ac} and \eqref{eq:Gamma-delta}.
\end{proposition}

\begin{proof}
$G$ is affine on $E_1$ and $E_2$, so the second-order part of $\Lc G$ vanishes
off $\Delta$, giving $\Lc G=\beta_1$ on $E_1$ and $\Lc G=\alpha\beta_2$ on
$E_2$. The singular part follows from \eqref{eq:D2G-measure} contracted with
$\tfrac12 a$, namely $\tfrac12\tr(a\,D^2G)=\tfrac{q}{2\sqrt{1+\alpha^2}}\sgD$.
Subtracting $\Lc G$ from $c+rG$ gives \eqref{eq:Gamma-ac}--\eqref{eq:Gamma-delta}.
\end{proof}

\subsection{Reflected It\^{o}--Tanaka formula}
\label{subsec:reflected-ito-tanaka}

Let $Y_t:=X_t^1-\alpha X_t^2$, so $G(X_t)=\alpha X_t^2+Y_t^+$. By the
one-dimensional Tanaka formula~\cite{revuz2013continuous,karatzas2014brownian}
$dY_t^+=\bone_{\{Y_t>0\}}\,dY_t+\tfrac12\,d\ell_t^0(Y)$, where $\ell^0(Y)$ is
the semimartingale local time of $Y$ at $0$ and $d\langle Y\rangle_t=q(X_t)\,dt$
with $q$ as in \eqref{eq:q-def}. Because $X$ is reflected,
$dY_t=\cdots+(\rho_1\!\cdot\!(1,-\alpha))\,dL_t^1+(\rho_2\!\cdot\!(1,-\alpha))\,dL_t^2$,
so the reflected It\^{o}--Tanaka formula for $G(X)$ contains coordinate-boundary
local-time terms.

On the open face $F_1=\{x_1=0,\ x_2>0\}$ one has $x_1=0<\alpha x_2$, so locally
$G=\alpha x_2$ and $\nabla G=(0,\alpha)$; hence
$\rho_1\!\cdot\!\nabla G=\alpha\theta_1$. On $F_2=\{x_2=0,\ x_1>0\}$ one has
$\alpha x_2=0<x_1$, so locally $G=x_1$ and $\nabla G=(1,0)$; hence
$\rho_2\!\cdot\!\nabla G=\theta_2$. The kink set $\Delta$ meets
$\partial\Rp$ only at the origin; by Assumption~\ref{ass:diffusion} the boundary
local times do not charge this corner, so the one-sided traces of $\nabla G$
above are relevant only $dL^i$-a.e.\ on the open faces.

\begin{lemma}[Reflected It\^{o}--Tanaka formula for $G$]
\label{lem:ito-tanaka-G}
For every $t\ge0$, after localization,
\begin{align}
\label{eq:ito-tanaka-G}
        e^{-rt}G(X_t)
        &=G(x)+\int_0^t e^{-rs}\big(\Lac G(X_s)-rG(X_s)\big)\,ds
        +\frac12\int_0^t e^{-rs}\,d\ell_s^0(Y)\nonumber\\
        &\quad+\alpha\theta_1\int_0^t e^{-rs}\,dL_s^1
        +\theta_2\int_0^t e^{-rs}\,dL_s^2+M_t^G,
\end{align}
where $M^G$ is a local martingale and
$\Lac G=\beta_1\bone_{E_1}+\alpha\beta_2\bone_{E_2}$.
\end{lemma}

\begin{proof}
Write $G(X_t)=\alpha X_t^2+Y_t^+$, apply the semimartingale Tanaka formula to
$Y^+$, insert the reflected SDE for $X$, and multiply by $e^{-rt}$. The
$dt$-terms give $\Lac G-rG$; the kink gives $\tfrac12 d\ell^0(Y)$; the
reflection terms are $(\rho_1\!\cdot\!\nabla G)\,dL^1+(\rho_2\!\cdot\!\nabla G)\,dL^2
=\alpha\theta_1\,dL^1+\theta_2\,dL^2$ on the open faces; the remaining
stochastic integral is $M^G$.
\end{proof}

\begin{remark}[Two equivalent singular descriptions]
\label{rem:local-time-vs-surface}
The diagonal singularity may be represented either by the surface measure
$-\tfrac{q}{2\sqrt{1+\alpha^2}}\sgD$ or by the local-time functional
$-\tfrac12\,d\ell^0(Y)$. These describe the same mechanism: the coefficient $q$
is already encoded in $\ell^0(Y)$ through $d\langle Y\rangle_t=q(X_t)\,dt$, and
one must not multiply the local-time term again by $q$.
\end{remark}

\begin{remark}[When the boundary terms vanish]
\label{rem:normal-reflection-vanishes}
The boundary contributions in \eqref{eq:ito-tanaka-G} are
$\alpha\theta_1\,dL^1$ and $\theta_2\,dL^2$. For normal reflection
($\theta_1=\theta_2=0$) both vanish, since on each face the gradient of $G$ is
parallel to that face and the inward normal derivative is zero. Thus, for the
pair $(G,\,\text{reflection})$, the boundary part of the stopping measure is
identically zero under normal reflection and is nonzero precisely when at least
one of $\theta_1,\theta_2$ is nonzero (equivalently, whenever a reflection
vector has a nonzero component tangent to its face). The remainder of the paper
is stated for general $\theta_1,\theta_2$ and specializes to normal reflection
by setting them to zero.
\end{remark}

\section{Total stopping measure and killed resolvents}
\label{sec:killed-resolvent}

\subsection{The total reflected stopping measure}
\label{subsec:total-measure}

Let $A^\Delta_t:=\tfrac12\ell_t^0(Y)$ be the diagonal local-time functional,
and let
\begin{equation}
\label{eq:A-boundary}
        A_t^\partial:=\alpha\theta_1 L_t^1+\theta_2 L_t^2
        =\sum_{i=1}^2\int_0^t(\rho_i\!\cdot\!\nabla G)(X_s)\,dL_s^i
\end{equation}
be the signed boundary functional. Define the total signed additive functional
\begin{equation}
\label{eq:A-total}
        A_t^{\Gtot}:=\int_0^t\Gac(X_s)\,ds-A_t^\Delta-A_t^\partial,
\end{equation}
and let $\Gtot$ be the signed smooth measure whose additive functional is
\eqref{eq:A-total}. In measure notation
\begin{equation}
\label{eq:Gamma-total}
        \Gtot=\Gac\,dx+\GD+\Gbd,
\end{equation}
with $\GD$ as in \eqref{eq:Gamma-delta} and $\Gbd$ the signed boundary Revuz
measure of $-A^\partial$, i.e. the measure $-\alpha\theta_1$ times boundary
occupation measure of $F_1$ plus $-\theta_2$ times that of $F_2$.

Combining \eqref{eq:ito-tanaka-G} with $\Gac=c+rG-\Lac G$ (so that
$\Lac G-rG=c-\Gac$) and \eqref{eq:A-total} yields the bridge identity in the
clean form
\begin{equation}
\label{eq:bridge}
        e^{-rt}G(X_t)-\int_0^t e^{-rs}c(X_s)\,ds
        =G(x)-\int_0^t e^{-rs}\,dA_s^{\Gtot}+M_t^G,
\end{equation}
where the running cost $c$ is already contained in $\Gac$. This is the form
used throughout.

\begin{assumption}[Smooth-measure admissibility]
\label{ass:smooth-measure}
The positive and negative variations of $\Gtot$ are smooth measures for the
reflected diffusion, and for every relatively compact open $O\subset\Rp$,
\[
        \E_x\!\left[\int_0^{\tau_O}e^{-rs}\,d|A_s^{\Gtot}|\right]<\infty,
        \qquad x\in O,\ \ \tau_O:=\inf\{t\ge0:X_t\notin O\}.
\]
\end{assumption}

\subsection{Killed resolvents}
\label{subsec:killed-resolvents}

For open $O\subset\Rp$ put $\tau_{O^c}:=\inf\{t\ge0:X_t\notin O\}$. For a
signed smooth measure $\mu=\mu^+-\mu^-$ with functional
$A^\mu=A^{\mu^+}-A^{\mu^-}$ define the killed $r$-resolvent on $O$,
\begin{equation}
\label{eq:killed-resolvent}
        R_r^O\mu(x):=\E_x\!\left[\int_0^{\tau_{O^c}}e^{-rs}\,dA_s^\mu\right],
\end{equation}
whenever the expectation is defined. When $O=\Cc$ we write $R_r^{\Cc}\mu$ and,
by Proposition~\ref{prop:finiteness}, $\tau_{\Cc^c}=\tau_{\Dc}$ up to polar
exceptional sets. The killed resolvent solves
\begin{equation}
\label{eq:killed-pde}
        (r-\Lc)R_r^O\mu=\mu\ \text{in }O,\qquad R_r^O\mu=0\ \text{on }\partial O,
\end{equation}
in the weak potential-theoretic sense
\cite{blumenthal2007markov,fukushima2011dirichlet}, with the reflection
condition retained
on the parts of $\partial O$ lying in $\partial\Rp$.

\begin{lemma}[Resolvent identity in killed domains]
\label{lem:resolvent-identity}
Let $O\subset\Rp$ be open and $\mu$ a signed smooth measure satisfying the
integrability in \eqref{eq:killed-resolvent}. Then $u=R_r^O\mu$ satisfies
\[
        u(x)=\E_x\!\big[e^{-r(t\wedge\tau_{O^c})}u(X_{t\wedge\tau_{O^c}})\big]
        +\E_x\!\left[\int_0^{t\wedge\tau_{O^c}}e^{-rs}\,dA_s^\mu\right]
\]
whenever both terms are finite.
\end{lemma}

\begin{proof}
Split the integral defining $R_r^O\mu$ at $t\wedge\tau_{O^c}$, use additivity
of $A^\mu$, and apply the strong Markov property to the shifted process.
\end{proof}

\subsection{Killed-resolvent representation of the value}
\label{subsec:killed-representation}

\begin{theorem}[Killed-resolvent representation]
\label{thm:killed-representation}
Assume Assumptions~\ref{ass:diffusion}--\ref{ass:smooth-measure}, that $V$ is
continuous, and that $\tau_{\Dc}$ is optimal for \eqref{eq:intro-value}.
Suppose
$\E_x[\int_0^{\tau_{\Dc}}e^{-rs}\,d|A_s^{\Gtot}|]<\infty$ for $x\in\Cc$ and
that the martingale in \eqref{eq:bridge} stopped at $\tau_{\Dc}$ is uniformly
integrable. Then, for every $x\in\Cc$,
\begin{equation}
\label{eq:V-killed}
        V(x)=G(x)-R_r^{\Cc}\Gtot(x),\qquad
        H(x)=-R_r^{\Cc}\Gtot(x).
\end{equation}
\end{theorem}

\begin{proof}
Fix $x\in\Cc$, $\tau=\tau_{\Dc}$. By optimality
$V(x)=\E_x[e^{-r\tau}G(X_\tau)-\int_0^\tau e^{-rs}c\,ds]$. Apply
\eqref{eq:bridge} at $t\wedge\tau$, localize, and let the localization parameter
tend to infinity; the integrability hypotheses give
$\E_x[M_{t\wedge\tau}^G]\to0$, whence
\[
        \E_x\!\Big[e^{-r\tau}G(X_\tau)-\int_0^\tau e^{-rs}c\,ds\Big]
        =G(x)-\E_x\!\Big[\int_0^\tau e^{-rs}\,dA_s^{\Gtot}\Big]
        =G(x)-R_r^{\Cc}\Gtot(x).
\]
Subtracting $G(x)$ gives the second identity.
\end{proof}

\begin{remark}[Sign and existence of an optimal time]
\label{rem:sign-existence}
Since $H=V-G\ge0$, \eqref{eq:V-killed} shows the killed potential
$R_r^{\Cc}\Gtot$ is nonpositive on $\Cc$, consistent with
$\Gtot=c+rG-\Lc G$ corrected by the singular terms. Optimality of $\tau_{\Dc}$
is not automatic; it requires $\Dc$ closed (guaranteed by continuity of $V$,
Proposition~\ref{prop:finiteness}) together with the standard regularity of the
stopping problem. We therefore state Theorem~\ref{thm:killed-representation}
conditionally and use it only as a representational counterpart to the
verification result of Section~\ref{sec:verification}.
\end{remark}

\subsection{Failure of the unrestricted reflected resolvent}
\label{subsec:unrestricted-fails}

Let $R_r^{\mathrm R}\mu(x):=\E_x[\int_0^\infty e^{-rs}\,dA_s^\mu]$ be the
unrestricted $r$-resolvent. A tempting replacement for \eqref{eq:V-killed} is
$V=G-R_r^{\mathrm R}(\Gtot\bone_{\Cc})$. Splitting at $\tau_{\Dc}$,
\begin{equation}
\label{eq:unrestricted-split}
        R_r^{\mathrm R}\big(\Gtot\bone_{\Cc}\big)(x)
        =\underbrace{\E_x\!\Big[\int_0^{\tau_{\Dc}}e^{-rs}\,dA_s^{\Gtot}\Big]}_{=\,R_r^{\Cc}\Gtot(x)}
        +\E_x\!\Big[\int_{\tau_{\Dc}}^\infty e^{-rs}\bone_{\Cc}(X_s)\,dA_s^{\Gtot}\Big].
\end{equation}

\begin{proposition}[Post-stopping occupation bias]
\label{prop:post-stopping-bias}
Assume the integrability of Assumption~\ref{ass:smooth-measure} and suppose that
both terms in \eqref{eq:unrestricted-split} are well defined. Then the
unrestricted representation $V=G-R_r^{\mathrm R}(\Gtot\bone_{\Cc})$ agrees with
the killed representation only when
\begin{equation}
\label{eq:no-post-stopping}
        \E_x\!\Big[\int_{\tau_{\Dc}}^\infty e^{-rs}\bone_{\Cc}(X_s)\,dA_s^{\Gtot}\Big]=0,
        \qquad x\in\Cc.
\end{equation}
In general this condition is not implied by reflection or by the stopping
structure, because the process is not absorbed on $\Dc$: a conservative
diffusion that enters $\Dc$ may subsequently return to $\Cc$.
\end{proposition}

\begin{proof}
The decomposition \eqref{eq:unrestricted-split} is the splitting of the
unrestricted functional at $\tau_{\Dc}$; the first term is $R_r^{\Cc}\Gtot$ by
Theorem~\ref{thm:killed-representation}. Equality of the two representations is
thus equivalent to the vanishing of the second term. The phenomenon is one of
non-absorption rather than of reflection per se: without an absorption
or invariance hypothesis on $\Dc$, the post-stopping occupation of $\Cc$ does
not vanish.
\end{proof}

\section{Free-boundary trace formulation}
\label{sec:free-boundary-trace}

The representation \eqref{eq:V-killed}, together with $V=G$ on $\Dc$, forces the
continuation-side trace of the killed potential to vanish:
\begin{equation}
\label{eq:trace}
        \Tr_{\partial\Cc}R_r^{\Cc}\Gtot=0,
        \qquad\text{equivalently}\qquad
        \Tr_{\partial\Cc}\big(G-R_r^{\Cc}\Gtot\big)=G.
\end{equation}
This is the nonlocal free-boundary condition for the reflected measure-valued
obstacle problem.

\begin{remark}[Meaning of the trace]
\label{rem:trace-meaning}
The trace in \eqref{eq:trace} is the continuation-side quasi-continuous trace of
the killed potential. When the killed potential has a continuous extension to a
regular boundary point, it reduces to the ordinary one-sided limit from $\Cc$.
For locally Lipschitz epigraph boundaries the classical expression
\begin{equation}
\label{eq:graph-trace}
        \lim_{\varepsilon\downarrow0}R_r^{\Cc}\Gtot\big(x_1,\bb(x_1)-\varepsilon\big)=0
\end{equation}
is used only at points where this one-sided trace exists; in general the trace
is taken in the quasi-continuous or fine-potential-theoretic sense, according
to the regularity available for $\bb$ and for the killed potential.
\end{remark}

\begin{remark}[Trace versus smooth fit]
\label{rem:trace-vs-smooth-fit}
At regular boundary points away from $\Delta$, additional regularity may turn
the trace condition into a more classical matching condition; we do not rely on
such a reduction.
The trace formulation is the appropriate global substitute precisely because of
the two singular mechanisms: the diagonal kink of $G$, where smooth fit
genuinely breaks down, and the boundary local-time terms generated by oblique
reflection. The condition \eqref{eq:trace} is a contact condition for the
killed potential, not a pointwise $C^1$- or $C^2$-matching condition.
\end{remark}

\section{Verification of epigraph candidates}
\label{sec:verification}

\subsection{Conditional epigraph geometry}
\label{subsec:epigraph}

In two-dimensional stopping problems the stopping set need not be a graph. We
impose graph structure through monotonicity of the stopping advantage.

\begin{assumption}[Vertical monotonicity]
\label{ass:vertical-monotonicity}
For every $x_1\ge0$ the map $x_2\mapsto H(x_1,x_2)$ is nonincreasing on
$[0,\infty)$.
\end{assumption}

\begin{theorem}[Conditional epigraph representation]
\label{thm:epigraph}
Assume $V$ continuous and Assumption~\ref{ass:vertical-monotonicity}. Define
\begin{equation}
\label{eq:b-def}
        \bb(x_1):=\inf\{x_2\ge0:H(x_1,x_2)=0\}\in[0,\infty],
        \qquad \inf\varnothing:=+\infty .
\end{equation}
Then $\bb:[0,\infty)\to[0,\infty]$ is lower semicontinuous and
\begin{equation}
\label{eq:DC-epigraph}
        \Dc=\{(x_1,x_2):x_2\ge\bb(x_1)\},\qquad
        \Cc=\{(x_1,x_2):0\le x_2<\bb(x_1)\}.
\end{equation}
\end{theorem}

\begin{proof}
Fix $x_1$. Since $H\ge0$ and $x_2\mapsto H(x_1,x_2)$ is nonincreasing,
$H(x_1,x_2)=0$ implies $H(x_1,y_2)=0$ for all $y_2\ge x_2$; hence each vertical
stopping section is an interval $[\bb(x_1),\infty)$ (empty if
$\bb(x_1)=+\infty$). Continuity of $V,G$ gives continuity of $H$, so $\Dc$ is
closed and $\bb$ is lower semicontinuous, and $H(x_1,\bb(x_1))=0$ when
$\bb(x_1)<\infty$. This yields \eqref{eq:DC-epigraph}.
\end{proof}

\begin{remark}[Conditional nature]
\label{rem:conditional}
Theorem~\ref{thm:epigraph} is not a consequence of convexity of $G$ or of
reflection: it follows from monotonicity of $H=V-G$, which in applications must
be established from the model primitives or imposed as part of a candidate
class. Allowing $\bb(x_1)=+\infty$ avoids an artificial nonempty-section
hypothesis.
\end{remark}

\subsection{Candidate epigraphs and the structural constraint}
\label{subsec:candidates}

Let $\bb:[0,\infty)\to[0,\infty]$ be locally Lipschitz on $\{\bb<\infty\}$, and
set
\begin{equation}
\label{eq:Db-Cb}
        \Dc_\bb:=\{x_2\ge\bb(x_1)\},\qquad \Cc_\bb:=\Rp\setminus\Dc_\bb,
        \qquad \tau_\bb:=\inf\{t\ge0:X_t\in\Dc_\bb\}.
\end{equation}
Define the killed potential and the candidate value
\begin{equation}
\label{eq:Pb-Ub}
        P_\bb(x):=R_r^{\Cc_\bb}\Gtot(x)
        =\E_x\!\Big[\int_0^{\tau_\bb}e^{-rs}\,dA_s^{\Gtot}\Big],
        \qquad
        U_\bb:=\begin{cases}G-P_\bb&\text{on }\Cc_\bb,\\[2pt]G&\text{on }\Dc_\bb.\end{cases}
\end{equation}

The following proposition makes explicit a constraint that the
superharmonicity requirement of Definition~\ref{def:admissible} below imposes on
the geometry; it should be read as part of the structural content of
admissibility rather than as an independently verifiable condition.

\begin{proposition}[Diagonal avoidance by the stopping interior]
\label{prop:diagonal-constraint}
Suppose $U_\bb=G$ on $\Dc_\bb$ and the measure-superharmonicity condition
$(r-\Lc)U_\bb+c\ge0$ holds on $\Rp$. Then
\[
        \Delta\cap\operatorname{int}(\Dc_\bb)=\varnothing.
\]
Consequently, for every $x_1$ with $\bb(x_1)<\infty$,
\begin{equation}
\label{eq:b-above-diagonal}
        \bb(x_1)\ge \frac{x_1}{\alpha}.
\end{equation}
\end{proposition}

\begin{proof}
On the interior of $\Dc_\bb$ one has $U_\bb=G$, hence
\[
        (r-\Lc)U_\bb+c=c+rG-\Lc G=\Gamma=\Gac\,dx+\GD,
\]
whose singular part is $\GD=-\tfrac{q}{2\sqrt{1+\alpha^2}}\sgD$, a strictly
negative measure on $\Delta$ by \eqref{eq:q-def}. If $\Delta$ met the interior
of $\Dc_\bb$, the measure $(r-\Lc)U_\bb+c$ would carry strictly negative
singular mass there, contradicting $(r-\Lc)U_\bb+c\ge0$. Hence
$\Delta\cap\operatorname{int}\Dc_\bb=\varnothing$. For the epigraph
$\Dc_\bb=\{x_2\ge\bb(x_1)\}$ the point $(x_1,x_1/\alpha)\in\Delta$ lies in the
interior of $\Dc_\bb$ iff $x_1/\alpha>\bb(x_1)$; excluding this gives
\eqref{eq:b-above-diagonal}.
\end{proof}

\begin{remark}[Interpretation]
\label{rem:diagonal-interpretation}
Constraint \eqref{eq:b-above-diagonal} says the free boundary lies on or above
the diagonal. The sign of the local-time term $+\tfrac12\,d\ell^0(Y)$ in
\eqref{eq:ito-tanaka-G} is consistent with the fact that, under
measure-superharmonicity, the diagonal cannot lie in the stopping interior. The
constraint is not an independent hypothesis on $\bb$ that can be checked before
fixing the geometry; it is forced by superharmonicity, which is why we record
it explicitly.
\end{remark}

\subsection{Admissible candidates and the verification theorem}
\label{subsec:verification-theorem}

\begin{definition}[Admissible epigraph candidate]
\label{def:admissible}
A locally Lipschitz epigraph boundary $\bb$ is admissible if, with
$P_\bb,U_\bb$ as in \eqref{eq:Pb-Ub}:
\begin{enumerate}[label=\textup{(A\arabic*)}]
\item\label{adm:finite} $P_\bb$ is finite q.e.\ in $\Cc_\bb$ and admits a
        quasi-continuous modification;
\item\label{adm:trace} $\Tr_{\partial\Cc_\bb}P_\bb=0$
        in the sense of Remark~\ref{rem:trace-meaning};
\item\label{adm:cont} $U_\bb$ extends continuously to $\Rp$;
\item\label{adm:contact} $U_\bb=G$ on $\Dc_\bb$ and $U_\bb>G$ on $\Cc_\bb$
        (contact and strict continuation);
\item\label{adm:major} $U_\bb\ge G$ on $\Rp$ (majorization);
\item\label{adm:neumann} $\rho_i\!\cdot\!\nabla U_\bb=0$ on $F_i$, $i=1,2$,
        in the trace sense (reflected Neumann compatibility);
\item\label{adm:growth} $|U_\bb(x)|\le C(1+\Psi(x))$ on $\Rp$, with $\Psi$ from
        Assumption~\ref{ass:growth};
\item\label{adm:ito} $U_\bb$ belongs locally to the generalized It\^{o} class,
        its distributional second derivatives having signed smooth-measure parts
        with locally integrable functionals;
\item\label{adm:super} $(r-\Lc)U_\bb+c\ge0$ on $\Rp$ as a signed smooth
        measure (measure-superharmonicity).
\end{enumerate}
\end{definition}

By Proposition~\ref{prop:diagonal-constraint}, any admissible $\bb$ satisfies
$\bb(x_1)\ge x_1/\alpha$.

Inside $\Cc_\bb$ the killed potential satisfies $(r-\Lc)P_\bb=\Gtot$ weakly, and
since $U_\bb=G-P_\bb$ with $\Gtot=c+rG-\Lc G$ (corrected by the singular and
boundary terms), the candidate value satisfies the continuation equation
\begin{equation}
\label{eq:candidate-continuation}
        \Lc U_\bb-rU_\bb-c=0\qquad\text{in }\Cc_\bb
\end{equation}
in the weak measure sense.

\begin{lemma}[Martingale property up to the candidate boundary]
\label{lem:candidate-martingale}
If $\bb$ is admissible, then
\[
        M_t^\bb:=e^{-r(t\wedge\tau_\bb)}U_\bb(X_{t\wedge\tau_\bb})
        -\int_0^{t\wedge\tau_\bb}e^{-rs}c(X_s)\,ds
\]
is a local martingale, and a true martingale after localization and passage to
the limit under \ref{adm:growth} and Assumption~\ref{ass:growth}.
\end{lemma}

\begin{proof}
Apply the generalized reflected It\^{o} formula \eqref{eq:reflected-ito} to
$U_\bb(X_{t\wedge\tau_\bb})$. By \ref{adm:neumann} the boundary terms
$\sum_i\int(\rho_i\!\cdot\!\nabla U_\bb)\,dL^i$ vanish, and by
\eqref{eq:candidate-continuation} the interior drift and all smooth-measure
finite-variation terms before $\tau_\bb$ cancel against
$-\int e^{-rs}c\,ds$. The remaining term is a local martingale; the Lyapunov
estimate gives uniform integrability after the standard localization.
\end{proof}

\begin{theorem}[Verification of an epigraph candidate]
\label{thm:verification}
Let $\bb$ be an admissible epigraph candidate in the sense of
Definition~\ref{def:admissible}, and $\tau_\bb$ as in \eqref{eq:Db-Cb}. Then
\[
        U_\bb=V\quad\text{on }\Rp,
\]
$\tau_\bb$ is optimal for \eqref{eq:intro-value}, and
\[
        \Dc_\bb=\{V=G\},\qquad \Cc_\bb=\{V>G\}.
\]
\end{theorem}

\begin{proof}
\emph{$U_\bb\ge V$.} Let $\tau$ be arbitrary. By \eqref{eq:reflected-ito} and
the superharmonicity \ref{adm:super}, the process
$e^{-r(t\wedge\tau)}U_\bb(X_{t\wedge\tau})-\int_0^{t\wedge\tau}e^{-rs}c\,ds$ is a
supermartingale after localization; the boundary terms vanish by
\ref{adm:neumann}. Using \ref{adm:growth} and letting the localization
parameter tend to infinity,
\[
        U_\bb(x)\ge\E_x\!\Big[e^{-r\tau}U_\bb(X_\tau)-\int_0^\tau e^{-rs}c\,ds\Big]
        \ge\E_x\!\Big[e^{-r\tau}G(X_\tau)-\int_0^\tau e^{-rs}c\,ds\Big],
\]
where the last inequality is \ref{adm:major}. Taking the supremum over $\tau$
gives $U_\bb(x)\ge V(x)$.

\emph{$U_\bb\le V$.} Take $\tau=\tau_\bb$. By
Lemma~\ref{lem:candidate-martingale},
$U_\bb(x)=\E_x[e^{-r\tau_\bb}U_\bb(X_{\tau_\bb})-\int_0^{\tau_\bb}e^{-rs}c\,ds]$.
Since $X_{\tau_\bb}\in\Dc_\bb$ and $U_\bb=G$ there by \ref{adm:contact},
\[
        U_\bb(x)=\E_x\!\Big[e^{-r\tau_\bb}G(X_{\tau_\bb})-\int_0^{\tau_\bb}e^{-rs}c\,ds\Big]\le V(x).
\]
Hence $U_\bb=V$ and $\tau_\bb$ is optimal. Finally \ref{adm:contact} gives
$\Dc_\bb=\{U_\bb=G\}=\{V=G\}$ and $\Cc_\bb=\{U_\bb>G\}=\{V>G\}$.
\end{proof}

\begin{remark}[Verification, not formal boundary solving]
\label{rem:verification-not-solving}
Theorem~\ref{thm:verification} does not assert that every solution of the trace
condition is optimal. The trace condition \ref{adm:trace} must be accompanied by
majorization, contact, strict continuation, reflected Neumann compatibility,
growth, generalized It\^{o} admissibility, and measure-superharmonicity. These
conditions are the analytical substitute for a formal smooth-fit calculation.
\end{remark}

\begin{corollary}[Uniqueness within the admissible class]
\label{cor:uniqueness}
If $\bb_1,\bb_2$ are admissible epigraph candidates, then
$U_{\bb_1}=U_{\bb_2}=V$ on $\Rp$, $\Dc_{\bb_1}=\Dc_{\bb_2}$, and
$\bb_1=\bb_2$ at every common point of continuity of the boundary.
\end{corollary}

\begin{proof}
By Theorem~\ref{thm:verification} both candidate values equal $V$, and
$\Dc_{\bb_j}=\{U_{\bb_j}=G\}=\{V=G\}$, so the stopping sets coincide; being
epigraphs, their boundary functions agree at common continuity points.
\end{proof}

\section{Conclusion}
\label{sec:conclusion}

We have formulated a reflected optimal stopping problem with the max-type payoff
$G=x_1\vee\alpha x_2$ as a measure-valued obstacle problem. The nonsmooth
obstacle produces a diagonal singular measure $\GD$ on $\Delta$, while oblique
reflection produces boundary local-time contributions; combining these with the
absolutely continuous interior gain yields the total stopping measure $\Gtot$.
For normal reflection the boundary component $\Gbd$ vanishes, but the
measure-valued formulation remains essential because the diagonal component
$\GD$ persists.

The correct potential representation of the value in the continuation region is
the killed resolvent $R_r^{\Cc}\Gtot$, not the unrestricted reflected resolvent,
because the process is not absorbed on the stopping set. Under a vertical
monotonicity hypothesis on $V-G$ the stopping set admits an epigraph
representation, and the killed-potential trace condition
$\Tr_{\partial\Cc}R_r^{\Cc}\Gtot=0$ becomes the natural free-boundary condition.
For admissible epigraph candidates, measure-superharmonicity prevents the
diagonal from entering the stopping interior; equivalently, at finite boundary
points one has \(\bb(x_1)\ge x_1/\alpha\).
The verification theorem shows that any admissible
epigraph candidate satisfying contact, strict continuation, reflected Neumann
compatibility, growth, the trace condition, and measure-superharmonicity
coincides with the value function, and that first entry into the candidate
stopping set is optimal.

The remaining analytical questions are the derivation of monotonicity of $V-G$
from primitive assumptions on $(\beta,\sigma,c,r,\alpha)$, finer regularity of
the killed-potential trace at the free boundary, and uniqueness of admissible
epigraph boundaries under minimal hypotheses.

\bibliography{reference}

\end{document}